\documentclass[11pt]{amsart}
\usepackage{pstricks,pst-plot}

\definecolor{Black}{cmyk}{0,0,0,1}
\definecolor{OrangeRed}{cmyk}{0,0.6,1,0}            
\definecolor{DarkBlue}{cmyk}{1,1,0,0.20}
\definecolor{myblue}{rgb}{0.66,0.78,1.00}
\definecolor{Violet}{cmyk}{0.79,0.88,0,0}
\definecolor{Lavender}{cmyk}{0,0.48,0,0}

\parskip=\smallskipamount

\newtheorem{theorem}{Theorem}[section]
\newtheorem{lemma}[theorem]{Lemma}
\newtheorem{corollary}[theorem]{Corollary}
\newtheorem{proposition}[theorem]{Proposition}

\theoremstyle{definition}

\newtheorem{remark}[theorem]{Remark}

\newcommand{\B}{\mathbb{B}}
\newcommand{\C}{\mathbb{C}}
\newcommand{\D}{\mathbb{D}}

\newcommand{\N}{\mathbb{N}}

\newcommand{\Aut}{\mathop{{\rm Aut}}}

\author[F. Bracci]{Filippo Bracci$^\ast$}

\address{F. Bracci: Dipartimento di Matematica\\
Universit\`{a} di Roma \textquotedblleft Tor Vergata\textquotedblright\ \\
Via Della Ricerca Scientifica 1, 00133 \\
Roma, Italy} \email{fbracci@mat.uniroma2.it}

\thanks{This article was written as part of the international research program "Several Complex Variables and Complex Dynamics" at the Centre for Advanced Study at the Norwegian Academy of Science and Letters in Oslo during the academic year 2016/2017.}
\thanks{$^\dag$ supported by the NRC grant number 240569}
\thanks{$^\ast$ Partially supported by GNSAGA of INDAM}

\author[J. E. Forn\ae ss]{John Erik Forn\ae ss$^\dag$}

\address{J. E. Forn\ae ss: Department of Mathematical Sciences, Norwegian University of Science and Technology, 7491 Trondheim, Norway. } \email{john.fornass@ntnu.no}

\author[E. F. Wold]{Erlend Forn\ae ss Wold$^\dag$}
\address{E. F. Wold: Department of Mathematics, University of Oslo, PO-BOX 1053 Blindern, 0316 Oslo, Norway.} \email{erlendfw@math.uio.no}

\newcommand{\bea}{\begin{eqnarray*}}
\newcommand{\eea}{\end{eqnarray*}}
\numberwithin{equation}{section}

\title[distances on strongly pseudoconvex and worm domains]{Comparison of invariant metrics and distances on strongly pseudoconvex domains and worm domains}

\begin{document}

\maketitle

\begin{abstract}
We prove that for a strongly pseudoconvex domain $D\subset\mathbb C^n$, the infinitesimal Carath\'{e}odory metric $g_C(z,v)$
and the infinitesimal Kobayashi metric $g_K(z,v)$ coincide if $z$ is sufficiently close to $bD$ and if $v$ is sufficiently close to 
being tangential to $bD$.  Also, we show that every two close points of $D$ sufficiently close to the boundary and whose difference is almost tangential to $bD$ can be joined by a (unique up to reparameterization) complex geodesic of $D$ which is also a holomorphic retract of $D$.

The same continues to hold if $D$ is a worm domain, as long as the points are sufficiently close to a strongly 
pseudoconvex boundary point.   We also show that a strongly pseudoconvex boundary point of a worm domain can be globally exposed; this has 
consequences for the behavior of the squeezing function.  
\end{abstract}

\section{Introduction}

For a domain $D\subset\mathbb C^n$, a point 
$z\in D$, and a vector $v\in T_zD=\mathbb C^n$, $v\neq 0$, we say that a holomorphic map $f_{z,v}:\triangle\rightarrow D$
is {\sl extremal} with respect to $(z,v)$, if $f_{z,v}(0)=z, f'(0)=\lambda v$, for some $\lambda>0$ and for any holomorphic $g:\triangle\rightarrow D, g(0)=z$, 
and $g'(0)=\tilde\lambda v$, we have that $|\tilde\lambda|\leq \lambda$.  A subset $S\subset D$ is a holomorphic retract, if there exists a holomorphic map $r:D\rightarrow D$ such that $r(D)= S$ and $r(z)=z$ for all $z\in S$.  For a bounded strongly convex domain $D$ of class $C^k, k\geq 3$, the following 
is due to Lempert \cite{Lem1, Lem2}: 
\begin{itemize}
\item[(1)]   the extremal map $f_{z,v}$ is unique, 
\item[(2)] $f_{z,v}$ extends to a $C^{k-2}$ map on $\overline\triangle$, and $f_{z,v}$ embeds $\overline\triangle$ into $\overline D$ with $f(b\triangle)\subset bD$,
\item[(3)] the corresponding extremal disc $S=f_{z,v}(\triangle)$ is a holomorphic retract of $D$,
\item[(4)] the extremal map $f_{z,v}$ is a complex geodesic, {\sl i.e.}, it is an isometry between the Poincar\'e distance of $\triangle$ and the Kobayashi distance of $D$,
\item[(5)] any two points $z,w\in D, z\neq w$, can be joined by a complex geodesic $f$, which is unique up to pre-composition with  automorphisms of $\triangle$.  Such a complex geodesic $f$ is an extremal map with respect to $(f(\zeta),f'(\zeta))$ for all $\zeta\in \triangle$ and its image is a holomorphic retract of $D$.
\end{itemize}

A straightforward consequence is that for a strongly convex domain $D$, the Carath\'{e}odory infinitesimal metric $g_C$ and the Kobayashi infinitesimal metric $g_K$ 
coincide, \emph{i.e.}, the quotient $Q(z,v)=g_C(z,v)/g_K(z,v)$ is identically equal to one.  This is no longer the case for non-convex domains, and
our main question here is to what extent it does continue to hold near strictly pseudoconvex boundary points of pseudoconvex domains in general.

For a domain $D\subset\mathbb C^n$ of class $C^2$ we let $\delta(z)$ denote the distance from $z$ to $b D$.
If $\delta(z)$ is small enough, there is a unique point $\pi(z)\in bD$ closest to $z$, and for any vector $v\in T_zD=\mathbb C^n$
there is an orthogonal decomposition $v=v_N+v_T$, with $v_T\in T_{\pi(z)}^{\C}bD$.  Our first two results 
are the following: 

\begin{theorem}\label{thm1}
Let $D\subset\mathbb C^n$ be a bounded strongly pseudoconvex domain of class $C^k$ for $k\geq 3$.  Then 
there exists $\epsilon>0$ such that the following hold: 
\begin{itemize}
\item[a)] for any $z\in D$ with $\delta(z)<\epsilon$ and 
$v\in T_zD\setminus\{0\}$ with $\|v_N\|<\epsilon\|v_T\|$, the extremal map $f_{z,v}$ satisfies (1), (2), (3) and (4) above.
\item[b)] for any $z,w\in D$ such that  $\max\{\delta(z),\delta(w), \mathrm{dist}(z,w)\}<\epsilon$  and 
$\|(w-z)_N\|<\epsilon \|(w-z)_T\|$, there exists a complex geodesic $f$ joining $z$ and $w$ which satisfies (5) above (in particular, $z, w$ are contained in a one-dimensional holomorphic retract of $D$).  As a consequence, the Kobayashi distance between $z$ and $w$ equals the Carath\'eodory  distance between $z$ and $w$.
\end{itemize} 
\end{theorem}
\begin{corollary}\label{cor1}
Let $D\subset\mathbb C^n$ be a bounded strongly pseudoconvex domain of class $C^k$ for $k\geq 3$.  Then 
there exists $\epsilon>0$ such that $Q(z,v)=1$ if $\delta(z)<\epsilon$ and if $\|v_N\|<\epsilon\|v_T\|$.
\end{corollary}

The theorem generalises a theorem of Kosinski \cite{Kosinski}, in which he shows that the invariant 
metrics coincide on \emph{some} open set.  Also, as pointed out by the referee, a version of the theorem for strongly pseudoconvex domains with boundaries
of class $C^6$, follows from the arguments in Section 4 in \cite{BurnsKrantz}. \

Considering domains more general than strictly pseudoconvex domains, a natural class of domains to consider 
is the class of Worm-domains (see Section \ref{worm} for the definition); these domains are pseudoconvex but without Stein neighbourhood bases, and 
provide counterexamples to many complex analytic problems of a global character.  

\begin{theorem}\label{thm2}
Let $\Omega_r$ be a Worm domain, and let $p\in b\Omega_r$ be a strongly pseudoconvex boundary point.  
Let $D$ be small open neighborhood of $p$.
Then 
there exists $\epsilon>0$ such that the following hold: 
\begin{itemize}
\item[a)] for any $z\in D\cap \Omega_r$ with $\delta(z)<\epsilon$ and 
$v\in T_zD\setminus\{0\}$ with $\|v_N\|<\epsilon\|v_T\|$, the extremal map $f_{z,v}$ satisfies (1), (2), (3) and (4) above.
\item[b)] for any $z,w\in D\cap \Omega_r$ such that  $\max\{\delta(z),\delta(w), \mathrm{dist}(z,w)\}<\epsilon$  and 
$\|(w-z)_N\|<\epsilon \|(w-z)_T\|$, there exists a complex geodesic $f$ joining $z$ and $w$ which satisfies (5) above (in particular, $z, w$ are contained in a one-dimensional holomorphic retract of $\Omega_r$). As a consequence, the Kobayashi distance between $z$ and $w$ equals the Carath\'eodory  between $z$ and $w$.

\end{itemize} 
\end{theorem}

The corresponding corollary continues to hold.    The main new theorem   needed in order to prove these results for Worm-domains is the following:

\begin{theorem}\label{thm3}
Let $p\in b\Omega_r$ be a strongly pseudoconvex boundary point of a Worm-domain $\Omega_r$.
Then for any $k\in\mathbb N\cup \{\infty\}$ there exists a $C^k$-smooth embedding $\phi:\overline{\Omega}_r\rightarrow\overline{\mathbb B^2}$
such that 
\begin{itemize}
\item[(i)] $\phi:\Omega_r\rightarrow\mathbb B^2$ is holomorphic, and 
\item[(ii)] $\phi(p)\in b\mathbb B^2$.
\end{itemize}
\end{theorem}
The point $\phi(p)$ is said to be {\sl exposed}.   Theorem \ref{thm2} was proved in \cite{DiederichFornaessWold} in the case
of strongly pseudoconvex domains.   The relevant difference between that and the present case, is that $\overline\Omega_r$
does not have a strongly pseudoconvex neighbourhood basis, and so we do not have access to a key ingredient in \cite{DiederichFornaessWold}, which 
is the existence of a certain compositional splitting of injective holomorphic maps, the proof of which relied on certain solution operators for
the $\overline\partial$-equation (see e.g. \cite{Forstnericbook}, 8.7).   Instead, we will prove the existence of such a splitting 
using H\"{o}rmander's $L^2$-theory, see Theorem \ref{thm:splitting} below.  \

A consequence of Theorem \ref{thm2} and the work \cite{DGZ2} is the following. 
\begin{theorem}\label{thm4}
Let $p\in b\Omega_r$ be a strongly pseudoconvex boundary point.  Then 
\begin{equation}
\lim_{z\rightarrow p}S_{\Omega_r}(z) = 1, 
\end{equation}
where $S_{\Omega_r}(z)$ denotes the squeezing function on $\Omega_r$.
\end{theorem}

The question whether any strongly pseudoconvex boundary point on a Worm domain can be exposed was 
raised in \cite{FornaessKim}.

\section{The Proof of Theorem \ref{thm1} and Theorem \ref{thm2}}

In this section we provide the proof of Theorem \ref{thm1} and Theorem \ref{thm2}. In order to do this, we need some preliminaries about real and complex geodesics.  \

For a domain $\Omega\subset\mathbb C^n$ and a piecewise $C^1$-smooth curve $\gamma:[0,1]\rightarrow\Omega$, we let $l_K(\gamma)$
denote the Kobayashi length of the image $\gamma$, \emph{i.e.}, 
\begin{equation}
l_K(\gamma):=\int_0^1 g_K(\gamma(t);\gamma'(t)) dt.
\end{equation}
For points $z,w\in\Omega$ we let $d_K(z,w)$ denote the induced distance between $z$ and $w$; the Kobayashi distance.   \

A {\sl real geodesic} (for the Kobayashi distance) is a piecewise $C^1$-smooth map $\gamma:[a,b]\to \Omega$ such that $d_K(\gamma(t),\gamma(s))=|t-s|$ for all $t,s\in [a,b]$. Here, $-\infty<a<b<+\infty$.  \

A {\sl complex geodesic} is a holomorphic map $\varphi:\triangle \to \Omega$ such that $d_P(\zeta, \eta)=d_K(\varphi(\zeta), \varphi(\eta))$ for all $\zeta, \eta\in \triangle$, where $d_P$ denotes the Poincar\'{e} distance.   \

That (1)-(4) are satisfied in Theorem \ref{thm1} is a consequence of the fact that boundary points of strongly pseudoconvex domains can be 
globally exposed (see Theorem \ref{Thm:immerso} below) and the following result due to 
X. Huang (see \cite[Corollary 1]{Hu}).

\begin{proposition}[Huang]\label{Huang}
Let $D\subset \C^n$ be a bounded strongly convex domain with $C^3$-smooth boundary. Let $p\in bD$ and let $U$ be an open neighborhood of $p$. Then there exist an open neighborhood $V$ of $p$ and $\epsilon_0>0$ such that for every $z\in V$ and for all $v\in \C^n\setminus\{0\}$ with $\|v_N\|< \epsilon_0 \|v_T\|$, the complex geodesic $\varphi:\triangle \to D$ such that $\varphi(0)=p$ and $\varphi'(0)=\lambda v$ for some $\lambda\in \C\setminus\{0\}$ satisfies $\varphi(\triangle)\subset U$. 
\end{proposition}

(As a matter of notation, if $D$ is a bounded domain with $C^2$-smooth boundary and $p\in b D$, we have an orthogonal splitting $\C^n=T^\C_p b D\oplus N$, with $N$ a one dimensional complex space. If $v\in \C^n$ is a vector, we let $v_N$ be the projection of $v$ on $N$ and $v_T$ the projection of $v$ on $T^\C_p bD$.) \

To get (5) we need to extend Huang's result to complex geodesics connecting the two points - see Proposition \ref{ext} below.   Finally, to get (1)-(5) in Theorem \ref{thm2}, we 
need to extend results on exposing points to Worm domains --- this will be done in Section \ref{worm}. \

We proceed to give some intermediate results before we prove Proposition \ref{ext}, and then we will prove Theorem \ref{thm1} and Theorem \ref{thm2}.   \

\begin{proposition}\label{Prop:geo-out}
Let $\Omega$ be a bounded strongly pseudoconvex domain with $C^2$-smooth boundary. Let $p\in b\Omega$ and let $U$ be an open neighborhood of $p$. Then there exist an open neighborhood $V$ of $p$ and a compact set $K\subset \Omega$ such that for every real geodesic $\gamma:[a,b] \to \Omega$ with $\gamma(a)\in V$ and $\gamma(b)\not\in U$, we have that $\gamma([a,b])\cap K\neq \emptyset$. 
\end{proposition}


\begin{proof} 
For $j\in\mathbb N$ we set $K_j:=\{z\in\Omega:\mathrm{dist}(z,b\Omega)\geq 1/j\}$.  Assume to get a contradiction that there exist
sequences of points $z_j\rightarrow p, w_j\in\Omega\setminus U$, and geodesics $\gamma_j$ connecting $z_j$ and $w_j$ with $\gamma_j\in\Omega\setminus K_j$.
Without loss of generality we may assume that $w_j\rightarrow q\in b\Omega\setminus U$.   Fix a point $z_0\in\Omega$.   By  \cite[Thm. 2.3.51]{Aba}) there exists 
a constant $C_1>0$ such that $d_K(z_0,z)\leq C_1 - \frac{1}{2}\log\delta(z)$ for $z\in\Omega$.
Set $\delta:=\mathrm{dist}(p,q)/6$.  By \cite[Thm. 2.3.54]{Aba}, 
there exist a constant $C_2\in\mathbb R$ and $0<\epsilon_1<\epsilon_2<\delta$ such that if $\mathrm{dist}(z,p)<\epsilon_1$ (resp. $q$) and if $\mathrm{dist}(w,p)\geq 2\epsilon_2$ (resp. $q$) then $d_K(z,w)\geq C_2- \frac{1}{2}\log\delta(z)$.   \

For each $j$ set $\tilde\gamma_j:=\gamma_j\setminus (B_{2\delta}(p)\cup B_{2\delta}(q))$.  For large enough $j$ we get that 
\begin{equation}
l_K(\gamma_j) \geq 2C_2 - \frac{1}{2}\log(\delta(z_j)) - \frac{1}{2}\log\delta(w_j) + l_K(\tilde\gamma_j),
\end{equation}
and 
\begin{equation}
\mathrm{dist}(z_j,z_0) + \mathrm{dist}(w_j,z_0)\leq 2C_1 - \frac{1}{2}\log(\delta(z_j)) - \frac{1}{2}\log\delta(w_j).
\end{equation}
Since the Euclidean length of $\tilde\gamma_j$ is longer than $\mathrm{dist}(p,q)/6$ for large $j$, and since $\tilde\gamma_j\subset\Omega\setminus K_j$ we 
have that $l_K(\tilde\gamma_j)\rightarrow\infty$ as $j\rightarrow\infty$.  So for large enough $j$ we have that $2C_2 + l_K(\tilde\gamma_j)\geq 2C_1$, contradicting 
the assumption that $\gamma_j$ is a geodesic.  
\end{proof}

 \begin{proposition}\label{Prop:convegence-geo-point}
Let $D\subset \C^n$ be a bounded strongly convex domain with $C^3$-smooth boundary. Let $\{\varphi_n\}$ be a sequence of complex geodesics parameterized so that $\delta(\varphi_n(0))=\sup_{\zeta\in \triangle} \delta(\varphi_n(\zeta))$. If there exists $p\in \partial D$ such that $\lim_{n\to \infty}\varphi_n(0)=p$, then  $\{\varphi_n\}$ converges uniformly on $\overline{\triangle}$ to the constant map $\overline{\triangle}\ni \zeta\mapsto p$.
\end{proposition}

\begin{proof}
It is enough to show that for every neighborhood $U$ of $p$ there exists $n_0\in \N$ such that $\varphi_n(\overline{\triangle})\subset U$ for all $n\geq n_0$. 
Assume to get a contradiction that this is not the case.  Then without loss of generality, we may assume that 
there exists an open neighborhood $U$ of $p$ such that $\varphi_{n}(\triangle)\not\subset U$ for all $n$.  \

Let $V$ and $K$ be given by Proposition \ref{Prop:geo-out}. For $n$ large enough, $\varphi_{n}(0)\in V$ and there exists $\zeta_n\in \triangle$ such that $\varphi_{n}(\zeta_n)\not\in U$. By pre-composing $\varphi_{n}$ with a rotation, we can assume that $\zeta_n\in (0,1)$.  The curve $\gamma_n:[0,\frac{1}{2}\log \frac{1+\zeta_n}{1-\zeta_n}]\ni t\mapsto \varphi_{n}(\tanh (t))$ is a real geodesic such that $\gamma_n(0)\in V$ and $\gamma_n(\frac{1}{2}\log \frac{1+\zeta_n}{1-\zeta_n})=\varphi_{n}(\zeta_n)\not\in U$. Hence, by Proposition \ref{Prop:geo-out}, we have that $\gamma_n([0,\frac{1}{2}\log \frac{1+\zeta_n}{1-\zeta_n}])\cap K\neq \emptyset$ for all $k$. This implies that there exists a constant $C>0$ such that  for all $n$ large enough we have that 
\[
\delta(\varphi_{n}(0))=\sup_{\zeta\in \triangle}\delta(\varphi_{n}(\zeta))\geq C,
\]
against the hypothesis $\lim_{n\to \infty}\varphi_{n}(0)=p$.
\end{proof}

The next proposition follows from  \cite[Proposition 1]{Hu} (see also \cite[Section 2]{BPT}):

\begin{proposition}[Huang]\label{Prop:Huang}
Let $D\subset \C^n$ be a bounded strongly convex domain with $C^3$-smooth boundary.  Let $\{\varphi_n\}$ be a sequence of complex geodesics in $D$ converging uniformly on compacta  to a complex geodesic $\varphi:\triangle \to D$. Then, $\{\varphi_n\}$ converges to $\varphi$ in the $C^1$-topology of $\overline{\triangle}$.
\end{proposition}

\begin{proposition}\label{ext}
Let $D\subset \C^n$ be a bounded strongly convex domain with $C^3$-smooth boundary.  Then for every $p\in bD$ and for every open neighborhood $U$ of $p$ there exist an open neighborhood $V\subset U$ of $p$ and $\epsilon>0$ such that for all $z,w\in V$ with $\|(z-w)_N\|<\epsilon \|(z-w)_T\|$  the complex geodesic $\varphi:\triangle \to D$ containing  $z$ and $w$ is contained in $U$.
\end{proposition}

\begin{proof}
Assume to get a contradiction the result is not true. Then, there exist an open neighborhood $U$ of $p$ and two sequences $\{z_n\}, \{w_n\}\subset D$ converging to $p$ with $\lim_{n\to \infty}\frac{z_n-w_n}{\|z_n-w_n\|}=v$ for some $v\in T^\C_p bD$, such that for every $n\in \N$,  the complex geodesic $\varphi_n:\triangle \to D$ which contains $z_n$ and $w_n$ satisfies $\varphi_n(\triangle)\not\subset U$. We can assume that $\varphi_n$ is parameterized in such a way that $\delta(\varphi_n(0))=\max_{\zeta\in \triangle}\delta(\varphi_n(\zeta))$ for all $n$. Up to subsequences, we can also assume that $\{\varphi_n\}$ converges uniformly on compacta to some holomorphic map $\varphi:\triangle \to \overline{D}$. By Proposition \ref{Prop:convegence-geo-point}, $\varphi(\triangle)\subset D$ and hence, since $\lim_{n\to \infty}d_K(\varphi_n(\zeta), \varphi_n(\eta))=d_K(\varphi(\zeta), \varphi(\eta))$ for all $\zeta, \eta\in \D$, it follows that $\varphi$ is a complex geodesic. By Proposition \ref{Prop:Huang}, $\{\varphi_n\}$ converges uniformly to $\varphi$ in $C^1$-norm on $\overline{\triangle}$ and $\{\varphi'_n\}$ converges uniformly on $\overline{\triangle}$ to $\varphi'$.

Since $\{z_n\}$ and $\{w_n\}$ are converging to $p$, it follows that there exists $\zeta\in \partial \triangle$ such that $\varphi(\zeta)=p$. Moreover, since $\lim_{n\to \infty}\frac{z_n-w_n}{\|z_n-w_n\|}=v$, then $\varphi'(\zeta)=\lambda v$ for some $\lambda\in \C$. However, since $v\in T^\C_p bD$, this contradicts the Hopf Lemma. 
\end{proof}

In order to prove Theorem \ref{thm1}, we need one more fact:

\begin{theorem}\label{Thm:immerso}
Let $D\subset\mathbb C^n$ be a bounded strongly pseudoconvex domain of class $C^k$, $k\geq 2$ and let $p\in bD$, or let $D$ be a Worm-domain, and $p\in bD$ a strictly pseudoconvex boundary point. Then there exists a bounded strongly convex domain $W$ with $C^k$-smooth boundary and a biholomorphism $\phi:D\to \C^n$ such that
\begin{enumerate}
\item $\phi$ extends as a diffeomorphism of class $C^k$ on $\overline{D}$,
\item $\phi(D)\subset W$,
\item there exists an open neighborhood $U$ of $\phi(p)$ such that $U\cap \phi(D)=U\cap W$.
\end{enumerate}
\end{theorem}

After applying \cite[Thm. 1.1]{DiederichFornaessWold} or Theorem \ref{thm3} below, this result follows from techniques in \cite{F}.  
For the convenience of the reader, we include here a complete proof.

\begin{proof}[Proof of Theorem \ref{Thm:immerso}]
By \cite[Thm. 1.1]{DiederichFornaessWold}  in case $D\subset\mathbb C^n$ is a bounded strongly pseudoconvex domain, 
or by Theorem \ref{thm3} in case  $D$ is a Worm-domain, there exists a $C^k$-smooth embedding embedding $\phi: \overline{D}\to \C^n$, holomorphic on $D$, 
$\phi(D)\subset \B^n:=\{z\in \C^n: \|z\|=1\}$, $\phi(p)\in b\B^n$ and $\phi(\overline{D}\setminus \{p\})\subset \B^n$.

Let $\widehat D$ be the convex hull of $\phi(D)$.  Then 
there exists an open neighbourhood $U$ of $\phi(p)$ such that $\widehat D\cap U=\phi(D)\cap U$.   Let $\psi$ be the signed distance function to $b\widehat D$.  Then 
$\psi$ is convex, and near the point $p$ it is strictly convex.  Let $\chi$ a non-negative smooth function, $\chi(x)=0$ near the 
origin, and $\chi=1$ near $\mathbb R^n\setminus\mathbb B^n$.  Then for small enough $\epsilon>0$ and small enough $\delta=\delta(\epsilon)>0$
the function $\tilde\psi=\psi - \delta\chi(x/\epsilon)$ is convex with $\tilde\psi=\psi$ near the origin,  $\tilde\psi<\psi$ for $\|x\|\geq\epsilon$, and 
$\tilde\psi$ is strictly convex for $\epsilon\leq\|x\|\leq 2\epsilon$.  Let $\tilde\chi$ be a non-negative smooth function which is $0$ near $\epsilon\overline{\mathbb B^n}$ and 
$1$ on $\mathbb R^n\setminus 2\epsilon\mathbb B^n$.  Smoothing, we may obtain a sequence of strictly convex functions $\tilde\psi_j$ converging to $\tilde\psi$ on a neighbourhood of $\overline{\{\tilde\psi<0\}}$, and the convergence is in $C^2$-norm on $\mathbb B^n_{2\epsilon}(0)$.  By Morse theory we may assume that $\nabla\tilde\psi_j$ is non-vanishing 
on $b{\{\tilde\psi_j<0\}}$.   So for sufficiently large $j$ we have that $\tilde\psi + \tilde\chi(\tilde\psi_j-\tilde\psi)$ defines a smoothly bounded strictly convex 
domain which agrees with $\phi(D)$ near $\phi(p)$.
\end{proof}

\begin{proof}[Proof of Theorem \ref{thm1} and Theorem \ref{thm2}:]
Let $\phi$ and $W$ be given by  Theorem \ref{Thm:immerso}.  The orthogonal splitting $\C^n=T^\C_pbD+N$ might not be preserved under $\phi$. Indeed, $d\phi_p(T^\C_pbD)=T^\C_{\phi(p)}b\phi(D)$ but $d\phi_p(N)$ might not be orthogonal to $T^\C_{\phi(p)}b\phi(D)$. Let $N'$ be the orthogonal complement of  $T^\C_{\phi(p)}b\phi(D)$ in $\C^n$. Then there exists a constant $C>0$ such that, if $v\in T_zD$ and $\|v_N\|<\epsilon \|v_T\|$ (in the splitting $\C^n=T^\C_pbD+N$), then  $\|{d\phi_z(v)}_N\|<C\epsilon \|d\phi_z(v)_T\|$ (in the splitting $\C^n=T^\C_{\phi(p)}b\phi(D)+N'$).

Therefore, without loss of generality, we can assume that there exists a bounded strongly convex domain $W$ with $C^3$ boundary such that $D\subset W$ and an open neighborhood $U$ of $p$ such that $D\cap U=W\cap U$. 

a) Let $z\in D$ and $v\in T_zD\setminus\{0\}$. Let $f_{z,v}$ be the extremal map for $W$ with respect to $(z,v)$. By Proposition \ref{Huang} there exist an open neighborhood $V$ of $p$ and $\epsilon_0>0$ such that $f_{z,v} (\triangle)\subset U$ provided $z\in V$ and $\|v_N\|<\epsilon_0 \|v_T\|$. Since $g_W(z,v)\leq g_D(z,v)$, it follows that $f_{z,v}$ is an extremal map for $D$ as well. It is also unique:  otherwise, since $D\subset W$, $W$ would have two different extremal maps with respect to $(z,v)$. 

Finally, according to Lempert's theory, every extremal disc of $W$ is a holomorphic retract of $W$, and, since $\phi(D)\subset W$ and $W\cap U=\phi(D)\cap U$, it follows that $f_{z,v}(\triangle)$ is a holomorphic retract of $\phi(D)$ as well. 
From this it is easy to see that $f_{z,v}$ is a complex geodesic of $D$. 

Hence, $f_{z,v}$ satisfies (1), (2), (3) and (4). By the compactness of $\overline{D}$, one can choose a uniform $\epsilon>0$ for every $p\in b D$ and the result is proved.

b) The argument is similar to a), using Proposition \ref{ext} instead of Proposition \ref{Huang}.
\end{proof}

\section{Worm Domains}\label{worm}

\subsection{The Worm}

We recall the definition of the Worm domains $\Omega_r$ from \cite{DiederichFornaess}.  
Let $\lambda:\mathbb R\rightarrow\mathbb R$ be a $\mathcal C^\infty$-smooth function such that 
\begin{itemize}
\item[a)] $\lambda(x)=0$ if $x\leq 0$,
\item[b)] $\lambda(x)>1$ if $x>1$,
\item[c)] $\lambda''(x)\geq 100\lambda'(x)$ for all $x$,
\item[d)] $\lambda'(x)>0$ if $x>0$, and 
\item[e)] $\lambda'(x)>100$ if $\lambda(x)>1/2$.
\end{itemize}
Let $r>1$. We set 
\begin{equation}
\rho_r(z,w):=|w+e^{i\log|z|^2}|^2-1+\lambda(\frac{1}{|z|^2}-1) + \lambda(|z|^2-r^2),
\end{equation}
and then $\Omega_r:=\{(z,w)\in\mathbb C^\ast\times \mathbb C:\rho_r(z,w)<0\}$.   We have that $\Omega_r$ is pseudoconvex, and 
that $b\Omega_r$ is strongly pseudoconvex away from the variety $Z_r:=b\Omega_r\cap \{w=0\}$.

\subsection{A Stein semi-neighbourhood basis of $\Omega_r$}

For $a, b, c>1$ we set 
\begin{equation}
\phi_{a, c}(z,w)=(z/a,  c e^{-i\log | a|^2}w)
\end{equation}
and
\begin{equation}
\Omega_{r, a, b, c}:=\phi_{ a, c}(\Omega_{ b r})
\end{equation}

\begin{lemma}\label{snbhb}
For $1< b<<2$ there exist $1< a, c< b$ such that $\Omega_r \subset \Omega_{r, a, b, c}$, 
and the boundaries of the two domains agree only along $\{w=0\}$.  Moreover, there exists $A>0$ such that 
for any $q\in b\Omega_r$ we have that 
\begin{equation}\label{distance1}
\mathrm{dist}(q,b\Omega_{r, a, b, c})\geq A\cdot\mathrm{dist}(q,\{w=0\})^2.
\end{equation}
\end{lemma}
\begin{proof}
Let us first set $ c=1$.  Then the defining function for $\Omega_{r, a, b, c}$ becomes
\begin{align*}
\tilde\rho(z,w) & = |e^{i\log | a|^2}w + e^{i\log a^2|z|^2}|^2 - 1 \\
& + \lambda(\frac{1}{ a^2|z|^2} - 1) + \lambda( a^2|z|^2 -  b^2 r^2) \\
& =  |w + e^{i\log|z|^2}|^2 - 1\\
& +   \lambda(\frac{1}{ a^2|z|^2} - 1) + \lambda( a^2|z|^2 -  b^2 r^2).
\end{align*}
Then the centers of the disc-fibers given by the projection $(z,w)\mapsto z$, remain the same, but the radii of some of them change.   The radii of 
the disc fibers of $\Omega_r$ start decreasing with decreasing $|z|$ at $|z|=1$, whereas the radii of the disc fibers of  $\Omega_{r, a, b, c}$
start decreasing at $|z|=1/ a$.   And for $|z|<1/ a$ their radii will always be strictly larger than those of $\Omega_r$.   \
Next the radii of the disc fibers of $\Omega_r$ are $1$ for $1<|z|\leq r$ and start decreasing at $|z|=r$.  The radii of the disc fibers of  $\Omega_{r, a, b, c}$
are 1 for $1/ a<|z|\leq (r b/ a)$, and start decreasing at $|z|=(r b/ a)$.   Furthermore, if $ a$ is chosen close enough to $1$, then the radii of the disc fibers of $\Omega_{r, a, b, c}$ as strictly larger than those of $\Omega_r$
for $|z|\geq (r b/ a)$.   For this it is enough that
\begin{align*}
0 & <  |z|^2-r^2 - ( a^2|z|^2 -  b^2r^2)\\
& \Leftrightarrow   a^2 - 1 < \frac{r^2( b^2-1)}{|z|^2},       
\end{align*}
which clearly holds for all relevant $z$ if $ a$ is close to 1.   After having fixed $ a$, we get the conclusions 
of the lemma except for \eqref{distance1}  by choosing $ c$ close enough to 1.   \

To see that we have \eqref{distance1} we apply the global change of coordinates $\psi(z,w):=(z,e^{-i\log|z|^2}w)$
defined on $\mathbb C^*\times\mathbb C$.  An application of $\psi$ only changes distances by a factor, so 
it suffices to consider instead the domains $\psi(\Omega_r)$ and $\psi(\Omega_{r, a, b, c})$.  For these
domains all disc fibers have the same centers.  
Moreover, 
the worst case to consider is when $1\leq |z| \leq r$, so it is enough to consider the two products
$\{1\leq |z| \leq r\}\times D_1$ and $\{1\leq |z| \leq r\}\times D_ c$, where $D_1$ is the disc of radius 
1 centered at the point 1, and $D_ c$ is the disc of radius $ c$ centered at the point $ c$.
And now we need only to show that for $w\in bD_1$ we have that $\mathrm{dist}(w,bD_ c)\geq\tilde A\cdot |w|^2$, which is easy to check. 
\end{proof}

\section{A splitting lemma for biholomorphic maps on worm domains}\label{splitting}

The following is the key technical ingredient in the proof of Theorem \ref{thm4}.  It was originally proved
by Forstneri\v{c} for strongly pseudoconvex Cartan pairs in complex Stein manifolds (See e.g. \cite{Forstnericbook}), and more 
recently on Stein spaces \cite{Forstneric}. 

\begin{theorem}\label{thm:splitting}
Suppose $\overline\Omega_r=A\cup B$, where $A$ and $B$ are closed sets such that 
\begin{itemize}
\item[(i)] $\overline{(A\setminus B)}\cap\overline {(B\setminus A)}=\emptyset$ (separation condition), 
\item[(ii)] $C:=A\cap B$ is a Stein compact, and 
\item[(iii)] $A\cap Z_r=\emptyset$.
\end{itemize}
Then for any open set $\tilde C\supset C$ and $k\in\mathbb N\cup\{\infty\}$,
there exist open sets $A'\supset A, \tilde C\supset C'\supset C$
such that for any injective holomorphic map $\gamma:C'\rightarrow\mathbb C^2$ sufficiently close to the identity, 
there exist
\begin{itemize}
\item[(a)] an injective holomorphic map $\alpha:A'\rightarrow\mathbb C^2$, and
\item[(b)] an injective map $\beta:B\rightarrow\mathbb C^2$, $\beta\in C^k(B)\cap\mathcal O(B\setminus Z_r)$,
\end{itemize}
such that $\beta|_C=\gamma\circ\alpha|_C$.  Moreover, $\alpha$ and $\beta$ may be chosen to depend 
continuously on $\gamma$, such that 
\begin{equation}
\|\alpha - \mathrm{id}\|_{A'}\rightarrow 0 \mbox{ and } \|\beta - \mathrm{id}\|_{C^k(B)}\rightarrow 0 \mbox{ as } \gamma\rightarrow\mathrm{id}.
\end{equation}
Finally, $\alpha$ and $\beta$ may be chosen to vanish to any given order along a given subvariety $W$ of $\mathbb C^2$
not intersecting $C$.
\end{theorem}

\begin{proof}
We start by doing some preparation for the proof of the theorem.   First we will construct
a sequence of domains $\Omega_j$  that decreases in a controlled way to a neighbourhood of $\Omega_r$ as in Lemma \ref{snbhb}.  
We will then be in a position to give a rough sketch of the proof.  
Then we will prepare for the vanishing on $W$.
Finally we need good estimates for solutions of linear Cousin problems on the $\Omega_j$'s.

\subsection{A decreasing sequence of domains}

For any closed set $K$ and $\nu>0$ we define
\begin{equation}
K(\nu):=\{q\in\mathbb C^2:\mathrm{dist}(z,K)\leq \nu\}.
\end{equation}
There exists $\nu_0>0$ small enough (See {\sl e.g.} the proof of Lemma 5.7.4 in \cite{Forstnericbook}) such that for any $0<\nu\leq \nu_0$ the pair $(A(\nu),B(\nu))$ still satisfies 
the separation condition corresponding to (i), and \begin{equation}
C(\nu)=A(\nu)\cap B(\nu)\subset U,
\end{equation}
where  $U$ is a fixed Stein neighborhood of $C$.

We set $C':=\mathrm{int}(C(\nu_0))$. 

Next choose $\tilde\Omega:=\Omega_{r,a,b,c}\subset A(\nu_0)\cup B(\nu_0)$ with $\Omega_{r,a,b,c}$ as in Lemma \ref{snbhb}.

If $\nu_0$ is small enough, we may let $\epsilon>0$ be small enough such that $\mathrm{dist}(A(\nu_0),\{w=0\})>\epsilon$,
and let $\psi$ be a nonnegative smooth function on $\mathbb C^2$, identically  equal
to $0$ on a neighbourhood of $|w|\leq \epsilon/4$ and identically equal to 1 on a neighbourhood of $|w|\geq \epsilon/2$. 
Let $\tilde\rho$ be a smooth defining function for $\tilde\Omega$ which 
is strictly plurisubharmonic near $\mathrm{Supp}(\psi)\cap b\tilde\Omega$.  Then for sufficiently small $s_0$ we have that 
\begin{equation}
\tilde\Omega_s:=\{\rho_s<0\} \mbox{ with }\rho_s := \tilde\rho - s\psi,
\end{equation}
is a pseudoconvex semi-neighbourhood of $\tilde\Omega$, for all $s<s_0$.  And if $a,b,c$ and $s_0$ are small enough, we 
have that $\tilde\Omega_{s_0}\subset A(\nu_0)\cup B(\nu_0)$. \

For $\widehat t<<\nu_0$ we set $A'=\tilde\Omega\cap A(\nu_0-\widehat t)$.  For an increasing sequence $t_j>0$ (to be constructed later)
with $t_j<\min\{\widehat t,s_0\}$ for all $j$, we define

\begin{itemize}
\item[(i)] $\Omega_j:=\tilde\Omega_{s_0-t_j}$, 
\item[(ii)] $A_j:=\Omega_j\cap A(\nu_0-t_j)$, 
\item[(iii)] $B_j:=\Omega_j\cap B(\nu_0-t_j)$, and 
\item[(iv)] $C_j:=A_j\cap B_j$.
\end{itemize}
We also set $t_0=0$ and $\Omega_0=\tilde\Omega_{s_0}$. \

\begin{lemma}\label{distance}
There exists a constant $k>0$ such that 
\begin{equation}
\mathrm{dist}(C_{j+1},\mathbb C^2\setminus C_j)\geq k\cdot (t_{j+1}-t_j).
\end{equation}
\end{lemma}
\begin{proof}
Note that $C_j=C(\nu_0-t_j)\cap\Omega_j$. 
Let $p\in bC_{j+1}$ and let  $q\in b C_{j}$.   Then we have four possibilities 
\begin{itemize}
\item[1)] $p\in bC(\nu_0-t_{j+1})$ and  $q\in bC(\nu_0-t_{j})$,
\item[2)]  $p\in bC(\nu_0-t_{j+1})$ and $q\in b\Omega_j$, 
\item[3)]  $p\in b\Omega_{j+1}$ and  $q\in bC(\nu_0-t_{j})$, or
\item[4)]   $p\in b\Omega_{j+1}$ and   $q\in b\Omega_j$,
\end{itemize}
and it is easy to check from these cases and the very definition of the sets that the statement holds.
\end{proof}

\subsection{Outline of the proof}

We will now give an outline of the proof.   Start by setting $c_0=c|_{C_0}$ where  $\gamma=\mathrm{id}+c$.  We will assume $\|c_0\|_{C_0}<\epsilon_0$ for 
some $\epsilon_0>0$ to be determined.   Seeking a compositional 
splitting $c_0=\beta\circ\alpha^{-1}$ we first find a linear splitting  $c_0=b_0-a_0$, where $b_0$ and $a_0$ has good 
$L^2$-estimates on $B_0$ and  $A_0$ respectively, depending on $\epsilon_0$.   Provided that $\epsilon_0,t_1,t_2$ are chosen carefully (to be explained in detail below), we 
will then get good sup-norm estimates for $b_0,a_0$ on $B_1$ and $A_1$ respectively, and then an estimate 
\begin{equation}
\|\beta_1^{-1}\circ\gamma\circ\alpha_1-\mathrm{id}\|_{C_2}\leq\epsilon_2,
\end{equation}
where $\alpha_1=\mathrm{id}+a_0, \beta_1=\mathrm{id}+b_0$, and $\epsilon_2$ is considerably smaller than $\epsilon_0$.
We then set $\gamma_2:=\beta_1^{-1}\circ\gamma\circ\alpha_1|_{C_2} = \mathrm{id}+c_2$, and repeat the process.  
Repeating the process indefinitely, provided we choose the sequences  $\{\epsilon_{2j}\}_{\j\in\mathbb N},\{t_k\}_{k\in\mathbb N}$
carfully and interdependently, we will obtain sequences $\alpha_{2j+1},\beta_{2j+1}$ of injective holomorphic maps 
defined on $A_{2j+2}$ and $B_{2j+2}$ respectively, such that 
\begin{equation}
\beta_{2k+1}^{-1}\circ\beta_{2k-1}\circ\cdot\cdot\cdot\circ\beta_1^{-1}\circ\gamma\circ\alpha_1\circ\cdot\cdot\cdot\circ\alpha_{2k-1}\circ\alpha_{2k+1}\rightarrow\mathrm{id}
\end{equation}
on  $C'$ as $k\rightarrow\infty$, and such that the compositions 
\begin{equation}
\alpha_1\circ\cdot\cdot\cdot\circ\alpha_{2k-1}\circ\alpha_{2k+1} \mbox{ and } \beta_1\circ\cdot\cdot\cdot\circ\beta_{2k-1}\circ\beta_{2k+1}
\end{equation}
converge to our desired maps $\alpha$ and  $\beta$ on $A'$ and  $B$ respectively.  

\subsection{Preparation for vanishing on $W$}

Let $f_1,...,f_N$ be entire holomorphic functions vanishing to a given order $k\in\mathbb N$ along $W$, with the property that they have no common zeroes 
on an open Stein neighbourhood $U\subset\tilde C$ of $C$.   Then by Cartan's division theorem there exist $g_j\in\mathcal O(U), j=1,...,N$, such that 
\begin{equation}
1 = \sum_{j=1}^N g_j\cdot f_j.
\end{equation} 
We will use these functions later in the proof.

\subsection{Estimates of splittings}

Note that for $k\in\mathbb N$ large enough, a suitable $k$th root $f_k(w)$ of  $1/w$ will satisfy $\mathrm{Re}(f_k)<0$ on $\Omega_r$, 
and $\mathrm{Re}(f_k(w))\leq  -C|w|^{-1/k}$.  So $\exp(-\mathrm{Re}(f_k(w)))\geq C_m\frac{1}{|w|^m}$ for any given $m$.  
We set 
\begin{equation}
\psi(z,w):= - \mathrm{Re}(f_k(w)) + |z|^2+|w|^2.
\end{equation}
By H\"{o}rmander \cite{Hormander} there exists a constant $C>0$, independent of $j$ and $m$, such that for any $\overline\partial$-closed $\omega\in L^2_{0,1}(\Omega_j,\psi)$
there exists  $h\in L^2(\Omega_j,\psi)$ with $\overline\partial h=\omega$, and
\begin{equation}
\int_{\Omega_j}|h|^2e^{-\psi}dV\leq C\cdot\int_{\Omega_j}|\omega|^2e^{-\psi}dV.
\end{equation}

\begin{lemma}\label{est1}
Let $s\in\mathbb N$.
There exists a constant $M_1>0$ such that the following holds.   Assume that we are given a holomorphic map $c_j:C_j\rightarrow\mathbb C^2$ with $\|c_j\|_{C_j}\leq\epsilon_j$, and set $t_{j+1}=t_j+\delta_{j+1}$.
Then there exist holomorphic maps $a_j:A_j\rightarrow\mathbb C^2$ and $b_j:B_j\rightarrow\mathbb C^2$, such that $c_j=b_j-a_j$, and such that 
\begin{equation}
\int_{A_j}|a_{j}|^2e^{-\psi}dV\leq M_1\epsilon_j^2, \mbox{ and }\int_{B_j}|b_{j}|^2e^{-\psi}dV\leq M_1\epsilon_j^2, 
\end{equation}
and 
\begin{equation}
\|a_j\|_{C_{j+1}}\leq \frac{M_1\epsilon_j}{\delta_{j+1}^2} \mbox{ and } \|b_j\|_{C_{j+1}}\leq \frac{M_1\epsilon_j}{\delta_{j+1}^2}.
\end{equation}
Moreover, the functions $a_j,b_j$ vanish to order $s$ along $W$.
\end{lemma}

\begin{proof}

Choose a nonegative smooth $\chi$ which is identically equal to zero in a neighbourhood of $\overline{A_{r_0}\setminus B_{r_0}}$ and 
identically equal to one  in a neighbourhood of $\overline{B_{r_0}\setminus A_{r_0}}$.  If $\widehat\tau$ is chosen small enough, the corresponding
separation conditions will hold with $r_0$ replaced by $j$ for all $j$.   \

Write 
\begin{equation}
c_j = c_j\cdot\sum_{i=1}^N g_i\cdot f_i = \sum_{i=1}^N f_i\cdot (g_i c_j).
\end{equation}
We obtain first smooth splittings of the maps $g_i c_j$ by setting $\tilde a_{ij}:=-\chi\cdot g_ic_j$ on $A_j$ and $\tilde b_{ij}:=(1-\chi)\cdot g_ic_j$
on $B_j$.  Now $\overline\partial\tilde a_{ij}=\overline\partial\tilde b_{ij}$ on $C_j$, and so we have a well defined closed $(0,1)$-form $\omega_{ij}$ on $\Omega_j$, whose 
sup-norm is proportional to that of $c_{j}$ independently of $j$. \

Now, the support of $\omega_{ij}$ is unformly bounded away from the singularity of the weight $\psi$, and so we obtain 
solutions $\overline\partial h_{ij}=\omega_{ij}$ with 

\begin{equation}
\int_{\Omega_j}|h_{ij}|^2e^{-\psi}dV\leq C'\epsilon_j^2, 
\end{equation}
where $C'$ is independent of $j$.  Setting $b_j:=\sum_{i=1}^N f_i\cdot (\tilde b_{ij}-h_{ij})$ and  $a_j:=\sum_{i=1}^N f_i\cdot (\tilde a_{ij}-h_{ij})$,
we obtain a holomorphic splitting $c_j=b_j-a_j$ with 

\begin{equation}
\int_{A_j}|a_{j}|^2e^{-\psi}dV\leq C''\epsilon^2_j, \mbox{ and }\int_{B_j}|b_{j}|^2e^{-\psi}dV\leq C''\epsilon^2_j
\end{equation}

And passing from the  $L^2$-estimate to a sup-norm estimate on $C_{j+1}$ using Lemma \ref{distance} we get 
\begin{equation}
\|a_j\|_{C_{j+1}}\leq \frac{C'''\epsilon_j}{\delta_{j+1}^2} \mbox{ and } \|b_j\|_{C_{j+1}}\leq \frac{C'''\epsilon_j}{\delta_{j+1}^2}
\end{equation}
where again $C'''$ is independent of $j$.
\end{proof}

The following lemma follows immediately from Lemma 8.7.4 in \cite{Forstnericbook}.

\begin{lemma}\label{est2}
There exists a constant $M_2>0$ such that the following holds.  Starting with the map $c_j$ on  $C_j$
from the previous lemma, assume that also $t_{j+2}=t_{j+1}+\delta_{j+1}$, and assume that 
$\frac{4M_1}{\delta_{j+1}^2}\epsilon_j<\delta_{j+1}$.  Set $\alpha_{j+1}=\mathrm{id}+a_{j}, \beta_{j+1}=\mathrm{id}+b_j$ and 
 $\gamma_{j}=\mathrm{id}+c_j$.  Set  $\gamma_{j+2}=\beta_{j+1}^{-1}\gamma_j\alpha_{j+1}=:\mathrm{id}+c_{j+2}$.  Then 
\begin{equation}
\|c_{j+2}\|_{C_{j+2}}\leq\frac{M_2}{\delta_{j+1}^{5}}\epsilon_j^2.
\end{equation}
\end{lemma}

\subsection{The proof of Theorem \ref{thm:splitting}}

Set $\gamma_0=\gamma|_{C_0}$, and write $\gamma_0=\mathrm{id}+c_0$.
To start an inductive construction, choose first $0<\delta_1<<1$ such that 
$4M_1\delta_1^{5}<1$, and $M_2\delta_1<1$ ($\delta_1$ will be further decreased several times throughout the proof).
Set $\epsilon_0=\delta_1^{8}$.  Set $\delta_2=\delta_1, t_1=\delta_1, t_2=\delta_1+\delta_2$.   By Lemma \ref{est1}
there exist $a_0:A_0\rightarrow\mathbb C^2,b_0:B_0\rightarrow\mathbb C^2$, such that  
$c_0=b_0-a_0$, and such that 
\begin{equation}
\|a_0\|_{C_{1}}\leq \frac{M_1\epsilon_0}{\delta_{1}^2} \mbox{ and } \|b_0\|_{C_{1}}\leq \frac{M_1\epsilon_0}{\delta_{1}^2}.
\end{equation}
Set $\alpha_1:=\mathrm{id}+a_0, \beta_1:=\mathrm{id}+b_0$.  Now 
\begin{equation}
\frac{4M_1}{\delta_1^2}\epsilon_0 = 4M_1\delta_1^{6}<\delta_1,
\end{equation}
so we are in the setting of Lemma \ref{est2}.  So writing $\gamma_2:=\beta_1^{-1}\gamma_0\alpha_1=\mathrm{id}+c_2$,
we get that 
\begin{equation}\label{eq}
\|c_{2}\|_{C_{2}}\leq\frac{M_2}{\delta_{1}^{5}}\epsilon_0^2=M_2\delta_1^{11}=M_2\delta_1(\delta_1^{5/4})^{8}<(\delta_1^{5/4})^8.
\end{equation}
This suggests how to define the sequences $\delta_j,\epsilon_j$ further to enable an inductive construction. 
Assume that we have defined $\delta_i$ for $i\leq 2k$ (which we have now done for $k=1$). 
Set $\delta_{2k+1}=\delta_{2k}^{5/4}$, 
and $\delta_{2k+2}=\delta_{2k+1}$.  Set $\epsilon_{2k}:=\delta_{2k+1}^{8}$, and for all $i$ set $t_i=\sum_{j=1}^i\delta_j$.   Then \eqref{eq} reads 
\begin{equation}
\|c_2\|_{C_2}\leq \delta_3^{8}=\epsilon_2.
\end{equation}

Now assume as our inductive hyposesis that we have constructed $\gamma_{2k}:C_{2k}\rightarrow\mathbb C^2, \gamma_{2k}=\mathrm{id}+c_{2k}$,
with $\|c_{2k}\|_{C_{2k}}\leq\epsilon_{2k}$ for $k=1,...,j$.  We complete the inductive step by repeating the above arguments essentially verbatim, only changing indices.  By Lemma \ref{est1}
there exist $a_{2j}:A_{2j}\rightarrow\mathbb C^2,b_{2j}:B_{2j}\rightarrow\mathbb C^2$, such that  
$c_{2j}=b_{2j}-a_{2j}$, and such that 
\begin{equation}\label{snest}
\|a_{2j}\|_{C_{2j+1}}\leq \frac{M_1\epsilon_{2j}}{\delta_{2j+1}^2} \mbox{ and } \|b_{2j}\|_{C_{2j+1}}\leq \frac{M_1\epsilon_{2j}}{\delta_{2j+1}^2}.
\end{equation}
Set $\alpha_{2j+1}:=\mathrm{id}+a_{2j}, \beta_{2j+1}:=\mathrm{id}+b_{2j}$.
Now 
\begin{equation}
\frac{4M_1}{\delta_{2j+1}^2}\epsilon_{2j} = 4M_1\delta_{2j+1}^{6}<\delta_{2j+1},
\end{equation}
so we are in the setting of Lemma \ref{est2}.  So writing $\gamma_{2j+2}:=\beta_{2j+1}^{-1}\gamma_{2j}\alpha_{2j+1}=\mathrm{id}+c_{2j+2}$,
we get that 
\begin{align*}
\|c_{2j+2}\|_{C_{2j+2}} & \leq\frac{M_2}{\delta_{2j+1}^{5}}\epsilon_{2j}^2=M_2\delta_{2j+1}^{11} \\
& =M_2\delta_{2j+1}(\delta_{2j+1}^{5/4})^{8} < \delta_{2j+3}^{8} = \epsilon_{2j+2}.
\end{align*}
This shows that the induction can go on indefinitely.   \

By the construction we see that, near $C$ we have that 
\begin{equation}
\lim_{j\rightarrow\infty}\beta(j)^{-1}\gamma_0\alpha(j):=\lim_{j\rightarrow\infty} \beta_{2j+1}^{-1}\circ\cdot\cdot\cdot\circ\beta_1^{-1}\circ\gamma_0\circ\alpha_1\circ\cdot\cdot\cdot\circ \alpha_{2j+1}=\mathrm{id}.
\end{equation}
So it remains to show that $\beta(j)$ converges to a smooth  injective map on $B$, and that $\alpha(j)$ converges to an injective map on $A'$.   \\

In the construction we defined $\alpha_{2j+1}=\mathrm{id}+a_{2j}$ and $\beta_{2j+1}=\mathrm{id}+b_{2j}$
on $A_{2j}$ and $B_{2j}$ respectively, and 
the estimates from Lemma \ref{est1} leading to the estimate \eqref{snest} were
\begin{equation}\label{intest}
\int_{A_{2j}}|a_{2j}|^2e^{-\psi}dV\leq M_1\epsilon_{2j}^2, \mbox{ and }\int_{B_{2j}}|b_{2j}|^2e^{-\psi}dV\leq M_1\epsilon_{2j}^2.
\end{equation}

We start by considering the simplest case, the convergence of the sequence $\alpha(j)$.  Note first that there exists 
an $\eta>0$ such that if $\delta_1$
is chosen small enough, then  $\mathrm{dist}(A',\mathbb C^2\setminus A_{2j})>\eta$ for all $j$.  Together with \eqref{intest}
this shows that each $\alpha_{2j+1}$ is injective holomorphic on $A'(\eta/2)$ provided $\delta_1$ is chosen 
small enough, and we get that the family $\alpha({i,j}):=\alpha_{2i+1}\circ\cdot\cdot\cdot\circ\alpha_{2j+1}, i<j$, is 
uniformly Lipschitz on $A'(\eta/4)$, with Lipschitz constant decreasing to zero with decreasing  $\delta_1$.  Then 
\begin{align*}
\|\alpha(j+1)(z)-\alpha(j)(z)\|_{A'} & = \|\alpha(j)(\alpha_{2j+1}(z))-\alpha(j)(z)\|_{A'}\\
& \leq C(\delta_1)\|a_{2j}(z)\|_{A'}\leq C(\delta_1)\delta_{2j+1},
\end{align*}
where the last inequality follows from \eqref{intest} provided $\delta_1$ is small enough.   This shows 
that $\alpha(j)$ converges to an injective holomorphic map on $A'$  provided $\delta_1$ is chosen small enough, 
and $\|c_0\|_{C'}\leq\epsilon_0=\delta_1^8$.  It is clear from the construction that the limit maps depend 
continuously on the input $\gamma$.  \

Now the same type of argument shows that for any $\epsilon>0$, the sequence $\beta(j)$ will converge uniformly 
to an injective holomorphic map on $B\cap\{|w|\geq \epsilon/2\}$, provided $\delta_1$ is chosen small enough.   
So it remains to show that the sequence $\beta(j)$ converges 
to an injective holomorphic map $\beta$ on $B\cap \{|w|\leq\epsilon\}$.

Notice first that all maps $\beta_{2j+1}$ are defined on $\Omega_{r,a,b,c}\cap\{|w|\leq\epsilon\}$.  So for all $q\in B\cap\{|w|\leq\epsilon\}$, the ball 
$B_{A|w|^2}(q)$ is contained in $B_{2j}$.  We get that
\begin{align*}
\int_{B_{A|w|^2}(q)}|b_{2j}|^2dV & \sim |w|^{m} \int_{B_{A|w|^2}(q)}|b_{2j}|^2 e^{-\psi}dV \\
& \leq |w|^{m}M_1\epsilon_{2j}^2.
\end{align*}
This shows that the sup-norm of $b_{2j}$ is comparable to $|w|^{m/2-4}\epsilon_{2j}$ and that 
the $C^1$-norm is comparable to $|w|^{m/2-5}\epsilon_{2j}$.  In particular we see that $\beta_{2j}$
is injective if $\delta_1$ was chosen small enough.   \

Next we need to show that the sequence $\beta(j)$ of compositions is well defined near $Z_r$.  We have that

\begin{align*}
\beta(j)(w) &  = \beta(j-1)(\beta_{2j+1}(q)) = \beta(j-1)(q + b_{2j}(q)),
\end{align*}
with $\|b_{2j}(q)\|\leq B\cdot |w|^k\delta_1^{8(5/4)^{j}}$ with $k=m/2-4$.
Because of \eqref{distance} the well definedness then follows from the following lemma (in which we drop restricting to odd indices), 
as long as we set $m\geq 12$.

\begin{lemma}
Let $B_k>0$ for $k\in\mathbb N, k\geq 2$ and let $\alpha>1$.  Then there exist $C_k>0, a(k)>0$ and $\delta_0>0$ such that 
the following holds.   For any $k$ define $f^k_1(x):=x + B_kx^k\delta^{\alpha}$, and define $f^k_j$  inductively by
\begin{equation}
f^k_{j+1}(x):=f^k_{j}(x + B_kx^k\delta^{\alpha^{j+1}}).
\end{equation}
Then $f^k_j(x)-x\leq C_k x^k\delta$ for all $x\in [0,a(k)]$ and $\delta\leq\delta_0$.
\end{lemma}
\begin{remark}
For simplicity we dropped the factor $8$ in the power - it would only make the estimates better. 
\end{remark}

\begin{proof}

We have that 
\begin{equation}\label{compder}
(f^k_{j+1})'(x):=(f^k_{j})'(x+B_kx^k\delta^{\alpha^{j+1}})\cdot (1+B_kkx^{k-1}\delta^{\alpha^{j+1}}).
\end{equation}

For $\tilde a(k)>0$ small, we prove by induction the statement $I_j$ that
\begin{equation}
(f^k_j)'(x)\leq \Pi_{i=1}^j (1+B_kk\delta^{\alpha^i}) \mbox{ for all } x\in [0, \tilde a(k)-B_k\tilde a(k)^k\sum_{i=1}^{j}\delta^{\alpha^i}].
\end{equation}
First of all $f_1'(x)\leq1+B_kk\delta^\alpha$ for all $x\in [0,1]$, so $I_1$ holds.  The map $x+B_kx^k\delta^{\alpha^{j+1}}$
maps the intervall $ [0, \tilde a(k)-B_k\tilde a(k)^k\sum_{i=0}^{j}\delta^{\alpha^i}]$ into the interval  $[0, \tilde a(k)-B_k\tilde a(k)^k\sum_{i=0}^{j-1}\delta^{\alpha^i}]$
and so by \eqref{compder} we get $I_{j+1}$.  This shows that the family $\{f^k_j\}$ is uniformly Lipschitz on $[0,a(k)]$, 
where $a(k)=\tilde a(k)-B_k\tilde a(k)^k\sum_{j=1}^\infty\delta^{\alpha^j}$.    So we get 

\begin{align*}
f_{j+1}(x)-x & = f_j(x+B_kx^k\delta^{\alpha^{j+1}}) - f_j(x) + f_j(x) - x \\
& = (\sum_{i=2}^j f_i(x + B_kx^k\delta^{\alpha^{i+1}})-f_i(x)) + f_1(x)-x \\
& \leq  C_kx^k\sum_{i=1}^j \delta^{\alpha^i} \leq C_k'x^k\delta.
\end{align*}

\end{proof}

So the limit $\beta=\mathrm{id}+b$ exists (near $Z_r$), and $\beta$ extends to the identity map on $Z_r$.  Finally, by 
the same scheme as above, we may control any finite $C^k$-norm up to $Z_r$ by increasing the integer $m$
in the application of the weight $\psi$.
\end{proof}

\section{Exposing points on Worm domains}

We will here briefly eplain how to prove Theorem \ref{thm3} following \cite{DiederichFornaessWold}, after having established Theorem \ref{thm:splitting} above.
The first steps in \cite{DiederichFornaessWold} provide an element $\Phi\in\Aut_{\mathrm{hol}}\mathbb C^2$ such that 
\begin{itemize}
\item[(1)] $\Phi(p)=0$,  
\item[(2)] $T_0(b\Phi(\Omega_r))=\{\mathrm{Re}(z)=0\}$,
\item[(3)] $\Phi(\overline\Omega_r)\cap\Gamma=\{0\}$, where $\Gamma:=\{w=\mathrm{Im}(z)=0, \mathrm{Re}(z)\geq 0\}$, and
\item[(4)] $b\Phi(\Omega_r)$ is strongly convex at $0$.
\end{itemize}
The existence of such a $\Phi$ relies only on the strict pseudoconvexity of $b\Omega_r$ at $p$, and no other global assumption about $\Omega_r$. \

Next set $U_\delta:=\{z\in\mathbb C:\mathrm{Re}(z)<0, |z|<\delta\}$ for some $0<\delta<<1$.  For a large $R>0$ we let, 
for each $j\in\mathbb N$, $g_j$ be a smooth map such that $g_j(z)=z$ for all  $z$ near $\overline U_\delta$, and such 
that $g_j$ embeds the interval  $I_j=[0,1/j]$ onto the interval  $[0,R]$.  By Mergelyan's theorem we may approximate 
$g_j$ by a holomorphic map $\tilde g_j$ in $C^1$-norm on $\overline U_\delta\cup I_j$, and we 
set $f_j(z):=\tilde g_j(z) + \overline{\tilde g_j(\overline z)}$.  Then by adding thinner and thinner strips around the $I_j$'s, 
we obtain a sequence of domains $U_{\delta,j}$, symmetric with respect to the $x$-axis, such that $f_j$ embeds $U_{\delta,j}$ onto a domain $V_j$, such that 
$f_j(I_j)=[0,R]$, and such that $R$ is a strongly convex globally exposed point for $V_j$, and we may achieve that $f_j\rightarrow\mathrm{id}$ uniformly on $\overline U_\delta$.  Moreover, we may achieve that $\overline U_{\delta,j}$ converges 
to the domain $U_\delta$ in the sense of Goluzin, and so letting $\psi_j:U_\delta\rightarrow U_{\delta,j}$ be 
the Riemann map symmetric with respect to the $x$-axis, we have that $\tilde f_j:=f_j\circ\psi_j$
converges uniformly to the identity map on $\overline U_\delta\setminus D_\mu(0)$ for any $\mu>0$.  \

Now for $0<\sigma<<1$ we set 
\begin{equation}
A_{\sigma}:=\{(z,w)\in\Phi(\Omega_r):\mathrm{Re}(z)\geq -\sigma , (z,w)\mbox{ close to the origin}\},
\end{equation}
and further $B_\sigma:=\overline{\Phi(\Omega_r)}\setminus A_{\sigma/2}$.  Then $(A_\sigma,B_\sigma)$ satisfies 
the hypotheses of Theorem \ref{thm:splitting}.  And setting $F_j(z,w):=(\tilde f_j(z),w)$, we have that 
$F_j\rightarrow\mathrm{id}$ uniformly on a fixed neighbourhood $\tilde C$ of $A_\sigma\cap B_\delta$.  We set $\gamma_j=F_j|_{\tilde C}$, and let $\alpha_j,\beta_j$ provide splittings as in Theorem \ref{thm:splitting}, such that $\alpha_j$ vanishes 
to order two at the origin.   Then 
the maps $\Psi_j$ defined as $\beta_j$ on $B_\sigma$ and $F_j\circ\alpha_j$ near $A_\sigma$ 
will for a sufficiently large $j$ have the property that $\phi_j=\Psi_j\circ\Phi$ embeds $\Omega_r$ into a sufficiently 
large ball with $(R,0)$ on its boundary (not necessarily centred at the origin), and 
with $\phi_j(p)=(R,0)$ a globally exposed point.   A scaling and a translation then gives the conclusion of the theorem.


\begin{thebibliography}{}


%
%
%
%
%
%
%
%
%
%
%
%
%

\bibitem{Aba} M. Abate, {\sl Iteration theory of holomorphic maps on taut manifolds}. {Research and Lecture Notes in Mathematics. Complex Analysis and Geometry}, {Mediterranean Press, Rende}, {1989}.

\bibitem{BB} Z. M. Balogh, M. Bonk, {\sl Gromov hyperbolicity and the Kobayashi metric on strictly pseudoconvex domains}. Comment. Math. Helv. 75 (2000), 504-533.

\bibitem{BPT} F. Bracci, G. Patrizio and S. Trapani, {\sl The pluricomplex Poisson kernel for strongly convex domains}.  Trans. Amer. Math. Soc., 361, 2, (2009), 979-1005.

\bibitem{BurnsKrantz}
Burns, D., and Krantz, S. G.; Rigidity of holomorphic mappings and a new Schwarz lemma at the boundary.  
\textit{J. Amer. Math. Soc.} {\ bf 7} (1994), no. 3, 661--676.  

\bibitem{Christ}
M. Christ,  {\sl Global $C^\infty$ irregularity of the $\overline\partial$-Neumann problem for worm domains}, 
J.  Amer. Math. Soc. {\bf 9} (1996), 1171--1185

\bibitem{DGZ2}
F. Deng, Q. Guan, L. Zhang, {\sl Properties of squeezing functions and global transformations of bounded domains}.
Trans. Amer. Math. Soc. {\bf 368} (2016), 2679--2696.


\bibitem{DiederichFornaess}
K. Diederich, J. E. Forn\ae ss, {\sl Pseudoconvex domains: an example with nontrivial Nebenh\"{u}lle}. Math. Ann.  {\bf 225}  (1977), no. 3, 275--292.


\bibitem{DiederichFornaessWold}
K. Diederich, K., J. E. Forn\ae ss, E. F. Wold, {\sl Exposing points on the boundary of a strictly pseudoconvex or a locally convexifiable domain of finite 1-type}.
J. Geom. Anal.  {\bf 24}  (2014),  no. 4, 2124--2134. 


\bibitem{F} 
J.E. Forn\ae ss, {\sl Embedding strictly pseudoconvex domains in convex domains.}
Amer. J. of Math. 98, (1976), 529--569.

\bibitem{FornaessKim}
J. E. Forn\ae ss, K.-T. Kim, {\sl 
Some problems}.  Complex analysis and geometry,  369--377, Springer Proc. Math. Stat., 144, Springer, Tokyo, 2015. 


\bibitem{FR}  F. Forstneri\v{c}, J.-P. Rosay, {\sl Localization of the Kobayashi metric and the boundary continuity of proper holomorphic mappings}. Math. Ann. 279 (1987), 239-252. 

\bibitem{Forstnericbook}
F. Forstneri\v{c}, {\sl Stein manifolds and holomorphic mappings. 
The homotopy principle in complex analysis}. Ergebnisse der Mathematik und ihrer Grenzgebiete. 3. Folge. 56. Springer, Heidelberg, 2011.

\bibitem{Forstneric}
F. Forstneri\v{c}, {\sl Noncritical holomorphic functions on Stein spaces}. \textit{J. Eur. Math. Soc.} (JEMS)  {\bf 18} (2016), no. 11, 2511--2543.


\bibitem{Hormander}
L. H\"{ormander}, {\sl An introduction to complex analysis in several variables}. Third edition. North-Holland Mathematical Library, 7. North-Holland Publishing Co., Amsterdam, 1990.


\bibitem{Hu} X. Huang, {\sl A preservation principle of extremal mappings near a strongly pseudoconvex point
and its applications}. Illinois J. Math. 38, 2 (1994), 283-302.

\bibitem{Kosinski}
Kosi\'{n}ski, L.; Comparison of invariant functions and metrics.  \textit{Arch. Math.} {\bf 102} (2014), 271--281.

\bibitem{Lem1} L. Lempert, {\sl La metrique de Koabayashi et la representation des domains sur la boule}, Bull. Soc. Math.  France, 109, (1981), 427-474.

\bibitem{Lem2} L. Lempert, {\sl Intrinsic distances and holomorphic retracts}. Complex Analysis and Applications
81, Sofia (1984), 341-364. 




\end{thebibliography}
\end{document}